%% file: main.tex
\documentclass[11pt]{article}

\usepackage{geometry}
\usepackage{graphics} 
\usepackage{epsfig} 
\usepackage{times} 
\usepackage{amsmath} 
\usepackage{amssymb}  
\usepackage{amsthm}

\usepackage{xcolor}

\newcommand{\E}{\mathbb{E}}

\newtheorem{theorem}{Theorem}

\newtheorem{lemma}{Lemma}

\newtheorem{proposition}{Proposition}
\newtheorem{corollary}{Corollary}

\title{Finite-time Identification of Stable Linear Systems\\
Optimality of the Least-Squares Estimator
}

\author{Yassir Jedra and Alexandre Proutiere
\thanks{This work was supported by the Wallenberg AI, Autonomous Systems and Software Program (WASP) funded by the Knut and Alice Wallenberg Foundation.}
\thanks{Y. Jedra and A. Proutiere are with the Division of Decision and Control Systems, School of Electrical Engineering and Computer Science, Royal institute of Technology (KTH), Stockholm, Sweden. Emails: \{{\it jedra@kth.se, alepro@kth.se}\}.}
}

\begin{document}

\maketitle
\thispagestyle{empty}
\pagestyle{empty}


\input{abstract}
\input{introduction}
\input{related}
\input{result}
\input{preliminaries}
\input{proof_new}

\bibliography{references,references2}{}
\bibliographystyle{IEEEtran}

\appendix
\input{appA}

\end{document}

%% file: abstract.tex

\begin{abstract}
We present a new finite-time analysis of the estimation error of the Ordinary Least Squares (OLS) estimator for stable linear time-invariant systems. We characterize the number of observed samples (the length of the observed trajectory) sufficient for the OLS estimator to be $(\varepsilon,\delta)$-PAC, i.e., to yield an estimation error less than $\varepsilon$ with probability at least $1-\delta$. We show that this number matches existing sample complexity lower bounds \cite{simchowitz2018learning,jedra2019sample} up to universal multiplicative factors (independent of ($\varepsilon,\delta)$ and of the system). This paper hence establishes the optimality of the OLS estimator for stable systems, a result conjectured in \cite{simchowitz2018learning}. Our analysis of the performance of the OLS estimator is simpler, sharper, and easier to interpret than existing analyses. It relies on new concentration results for the covariates matrix.  
\end{abstract}

%% file: introduction.tex

\section{Introduction}

We investigate the canonical problem of identifying Linear Time Invariant (LTI) systems of the form:
\begin{equation}\label{eq:dynamics}
x_0=0 \ \ \ \hbox{ and } \ \ \ \forall t\ge 0,\  x_{t+1} = A x_{t} + \eta_{t+1}.
\end{equation}
$x_t \in \mathbb{R}^d$ denotes the state at time $t$, $A \in \mathbb{R}^{d\times d}$ is initially unknown but stable (i.e., its spectral radius satisfies $\rho(A) < 1$), and $(\eta_t)_{t\ge 1}$ are i.i.d. sub-gaussian zero-mean noise vectors with covariance matrix $I_d$ (for simplicity). The objective is to estimate the matrix $A$ from an observed trajectory of the system. Most work on this topic is concerned with the convergence properties of some specific estimation methods (e.g. ordinary least squares (OLS), ML estimator) \cite{Goodwin:1977:DSI,Ljung:c1,Ljung:c2}. Recently however, there have been intense research efforts towards understanding the finite-time behavior of classical estimation methods \cite{rantzer2018,faradonbeh:2018:c2, simchowitz2018learning, sarkar:2018, oymak:2018, sarkar:2019}. The results therein aim at deriving bounds on the {\it sample complexity} of existing estimation methods (mostly the OLS), namely at finding upper bounds on the number of observations $\tau$ sufficient to identify $A$ with prescribed levels $(\varepsilon,\delta)$ of accuracy and confidence\footnote{Throughout the paper, $\|\cdot\|$ denotes the euclidian norm for vectors, and the operator norm for matrices.}: $\mathbb{P}(\|A_t-A\|\le \varepsilon)\ge 1-\delta$ for any $t\ge \tau$ if $A_t$ denotes the estimator of $A$ after $t$ observations. Lower bounds of the sample complexity (valid for {\it any} estimator) have been also derived in \cite{simchowitz2018learning,jedra2019sample}. Sample complexity upper bounds appearing in the aforementioned papers are often hard to interpret and to compare. The main difficulty behind deriving such bounds stems from the fact that the data used to estimate $A$ is correlated (we observe a single trajectory of the system). In turn, existing analyses rely on involved concentration results for random matrices \cite{mendelson:2014, vershynin:2012} and self-normalized processes \cite{pena2008self}. The most technical and often tedious part of these analyses deals with deriving concentration results for the covariates matrix $X$ defined as $X^\top=(x_1 ,\ldots, x_t)$, and the tightness of these results directly impacts that of the sample complexity upper bound (refer to \textsection \ref{sec:related} for more details).

In this paper, we present a novel analysis of the error $\|A_t-A\|$ of the OLS estimator. In this analysis, we derive tight concentration results for the entire spectrum of the covariates matrix $X$. To this aim, we first show that the spectrum of $X$ can be controlled when
 $\| (XM)^\top XM - I_d \|$ is upper bounded, where $M=\left(\sum_{s=0}^{t-1} \Gamma_s(A)\right)^{-\frac{1}{2}}$ and $\Gamma_s(A)=\sum_{k=0}^s A^k(A^k)^\top$ is the finite-time controllability gramian of the system. We then derive a concentration inequality for $\| (XM)^\top XM - I_d \|$ by (i) expressing this quantity as the supremum of a {\it chaos process} \cite{krahmer2014suprema}; and (ii) applying Hanson-Wright inequality \cite{rudelson2013hanson} and an $\epsilon$-net argument to upper bound this supremum.

Our main result is simple and easy to interpret. For any $(\varepsilon, \delta)>0$, we establish that the OLS estimator is $(\varepsilon, \delta)$-PAC after $t$ observations, i.e., $\mathbb{P}(\|A_t-A\|\le \varepsilon)\ge 1-\delta$, provided that:
\begin{equation}\label{eq:sc}
\lambda_{\min}(\sum_{s=0}^{t-1} \Gamma_s(A)) \ge c\max\{{1\over \varepsilon^2},\mathcal{J}(A)^2\} (\log(1/\delta) + d),
\end{equation}
where $\lambda_{\min}(W)$ denotes the smallest eigenvalue of $W$, $c$ is a universal constant, and $\mathcal{J}(A)=\sum_{s\ge 0}\|A^s\|$ depends on $A$ only and is finite when $\rho(A)<1$ (${\cal J}(A)\le 1/(1-\|A\|)$ if $\|A\|<1$).

In \cite{simchowitz2018learning}, the authors have shown that an estimator can be $(\varepsilon,\delta)$-PAC after $t$ observations uniformly over the set ${\cal L}=\{A : \exists O\in \mathbb{R}^{d\times d}, A=\rho O \hbox{ and } O^\top O=I_d \}$ of scaled-orthogonal matrices for some fixed $\rho > 0$ only if the following necessary condition holds: for $A\in {\cal L}$,
\begin{equation}\label{eq:need}
\lambda_{\min}(\sum_{s=0}^{t-1} \Gamma_s(A)) \ge c{1 \over \varepsilon^2}(\log(1/\delta)+d),
\end{equation}
for some universal constant $c>0$. Hence our sufficient condition (\ref{eq:sc}) cannot be improved when the accuracy level $\varepsilon$ is sufficiently low, i.e., when $\varepsilon=O({\cal J}(A)^{-1})$. This result was actually conjectured in \cite{simchowitz2018learning}. More recently in \cite{jedra2019sample}, we also proved that for any arbitrary $A$, a necessary condition for the existence of a $(\varepsilon, \delta)$-PAC estimator after $t$ observations is $\lambda_{\min}(\sum_{s=0}^{t-1} \Gamma_s(A)) \ge c{1 \over \varepsilon^2}\log(1/\delta)$, but we believe that this sample complexity lower bound can be improved to \eqref{eq:need} (for arbitrary matrix $A$, not only for orthogonal matrices). Anyway, when $(\varepsilon, \delta)$ approach 0, and more precisely when $\varepsilon=O({\cal J}(A)^{-1})$ and $\delta=O(e^{-d})$, the necessary condition derived in \cite{jedra2019sample} and the sufficient condition \eqref{eq:sc} are identical. They would be also identical for any $\delta>0$, if we manage to show that the sample complexity lower bound \eqref{eq:need} holds for any matrix.

An other way of presenting our results is to consider the rate at which the estimation error decays with the number of observations $t$. We obtain: with probability at least $1-\delta$,
\begin{equation}\label{eq:rate}
\|A_t-A\| \le c   \sqrt{\log(1/\delta) + d\over \lambda_{\min}(\sum_{s=0}^{t-1} \Gamma_s(A))}.
\end{equation}
One can readily check that $\lambda_{\min}(\sum_{s=0}^{t-1} \Gamma_s(A))\ge t$. Hence, $\|A_t-A\|$ decays as $\sqrt{\log(1/\delta)+d\over t}$, which is the best possible decay rate. In \textsection \ref{sec:related}, we compare our result to those of existing analyses of the OLS estimator. Our result is tighter than state-of-the-art results, and it is derived using much simpler arguments.

\medskip
\noindent
{\bf Notations.} Throughout the paper, $\|W\|$ denotes the operator norm of the matrix $W$, and the Euclidian norm of a vector $w$ is denoted either by $\| w\|$ or $\| w\|_2$. The unit sphere in $\mathbb{R}^d$ is denoted by $S^{d-1}$. The singular values of any real matrix $W \in \mathbb{R}^{m \times d}$ with $m \ge d$ are denoted by $s_k(W)$ for $k =1, \dots, d$, arranged in a non-increasing order. For any square matrix $W$, $W^\dagger$ represents its Moore-Penrose pseudo-inverse. The vectorization operation $\textrm{vec}(\cdot)$ acts on matrices by stacking its columns into one long column vector.
Next, we recall the definition of sub-gaussian random variables and vectors. The $\psi_2$-norm of a r.v. $X \in \mathbb{R}$ is defined as $\Vert X \Vert_{\psi_2} = \inf\{ K\ge 0: \mathbb{E}[\exp(X^2/K^2)] \le 2 \}$. $X$ is said {\it sub-gaussian} if $\Vert X \Vert_{\psi_2}<\infty$. Now a random vector $X \in \mathbb{R}^d$ is sub-gaussian if its one-dimensional marginals $X^\top x $ are sub-gaussian r.v.'s for all $x \in \mathbb{R}^d$. Its $\psi_2$-norm is then defined as
$
\Vert X \Vert_{\psi_2} = \sup_{x \in S^{d-1}} \Vert X^\top x \Vert_{\psi_2}.
$
Note that a gaussian vector with covariance matrix equal to the identity is sub-gaussian, and its $\psi_2$-norm is 1. A vector $X\in \mathbb{R}^d$ is {\it isotropic} if
$
\E[ (X^\top x)^2] = \Vert x \Vert^2 ,
$
for all $x \in \mathbb{R}^d$ or equivalently if $\E XX^\top =I_d$.

Finally, we introduce notations specifically related to our linear system. Let $X$ and $E$ be the matrices defined by $X^\top=(x_1 ,\ldots, x_t)$ and $E^\top=(\eta_2,\ldots,\eta_{t+1})$. We further define the noise vector $\xi$ by $\xi^\top=(\eta_2^\top, \ldots, \eta_{t+1}^\top)$.

%% file: related.tex

\section{Related work}\label{sec:related}

The finite-time analysis of the OLS estimator has received a lot of attention recently, see \cite{faradonbeh2018finite}, \cite{simchowitz2018learning}, \cite{sarkar2019near} and references therein.

In \cite{faradonbeh2018finite}, the authors prove that the OLS estimator is $(\varepsilon,\delta)$-PAC if the observed trajectory is longer than $\frac{1}{\varepsilon^2}\log\left(\frac{1}{\delta}\right)^{3} c_1(A) d\log(d)$. In this upper bound, the constant $c_1(A)$ depends on $A$ in a complicated manner. The bound does not exhibit the right scaling in $\log(1/\delta)$ and $\varepsilon$. It also has a worse scaling with the dimension $d$ than our bound.

The main result in \cite{simchowitz2018learning} (Theorem 2.1) states that the OLS estimator is $(\varepsilon,\delta)$-PAC after $t$ observations under the following condition:
$$
t\ge {1\over \varepsilon^2\lambda_{\min}(\Gamma_k(A))}(d\log(d/\delta)+\log\det(\Gamma_t(A)\Gamma_k(A)^{-1})),
$$
for some $k$ satisfying
$$
{t\over k}\ge c' (d\log(d/\delta)+\log\det(\Gamma_t(A)\Gamma_k(A)^{-1})).
$$
This result is difficult to interpret, and choosing $k$ to optimize the bound seems involved. The authors of \cite{simchowitz2018learning} present a simplified result in Corollary 2.2, removing the dependence in $k$. However, the result requires that $t\ge T_0$, and we can show that $T_0$ actually depends on $A$ (the authors do not express this dependence). Corollary A.2 presented in the appendix of \cite{simchowitz2018learning} is more explicit, and states that OLS is $(\varepsilon,\delta)$-PAC if
$$
t \ge c( d \log\left( d \textrm{cond}(S) t/\delta\right) + \sum_{l}b_l^2 \log t ),
$$
where $b_1, b_2, \dots$ are the block sizes of the Jordan decomposition of $A$, and $S$ is the diagonalizing matrix in this decomposition. Note that the term $\sum_{l} b_l^2$ may be in the worst case of the order $d^2$. Compared to our sample complexity upper bound \eqref{eq:sc}, the above bound has a worse dependence in the dimension $d$. It has the advantage not to have the term ${\cal J}(A)$, but this advantage disappears when $\varepsilon \le c'\mathcal{J}(A)^{-1}$. The analysis of  \cite{simchowitz2018learning} relies on the decomposition of the estimation error $\|A_t-A\|\le \| X^\dagger\| \| U^\top E\|$, where $U$ is obtained from the singular value decomposition $X = U \Sigma V^\top$. Deriving an upper bound of $\| X^\dagger\|$ is then the most involved step of the analysis. To this aim, the authors adapt the so-called small ball method \cite{mendelson2014learning} (which typically requires independence), and introduce the Block Martingale Small Ball condition (BMSB) (which indeed requires the introduction of the term $k$ in the result).

The authors of \cite{sarkar2019near} use the same decomposition of $\|A_t-A\|$ as ours. This decomposition is $\Vert A_{t} - A \Vert \le \Vert E^\top X ((X^\top X)^\dagger)^\frac{1}{2}  \Vert \Vert ((X^\top X)^\dagger)^\frac{1}{2} \Vert$. The first term corresponds to a self-normalized process, and can be analyzed using the related theory \cite{pena2008self}. The analysis of the second term again requires to control the singular values of $X$. In turn, the analysis of $s_d(X)$ presented in \cite{sarkar2019near} is involved, and unfortunately leads to bounds that are not precise in the system parameters $A$ and the dimension $d$. Overall, the upper bound of the sample complexity proposed in \cite{sarkar2019near} is $O(C(d){1\over \varepsilon^2}\log(1/\delta))$. The dependence in $d$ is not explicit, and that in $A$ is unknown and hidden in constants.

%% file: result.tex

\section{Main result}

Consider the linear system (\ref{eq:dynamics}), and assume that the noise vectors are i.i.d. isotropic vectors with independent coordinates of $\psi_2$-norm upper bounded by $K$. We observe a single trajectory of the system $(x_1,\ldots,x_{t+1})$, and builds from these observations the OLS estimator: $A_{t} = \underset{A \in \mathbb{R}^{d\times d}}{\arg\min} \sum_{s=0}^{t} \Vert x_{s+1} - A x_s \Vert^2$. $A_t$ enjoys the following closed form expression:
\begin{equation}
  A_{t}  = \left( \sum_{s=0}^{t} x_{s+1} x_s^\top  \right) \left( \sum_{s=0}^t x_s x_s^\top \right)^{\dagger}.
\end{equation}
The estimation error is also explicit:
\begin{equation}\label{eq:estimation error}
      A_{t}-A  = \left( \sum_{s=1}^t \eta_{s+1} x_s^\top  \right) \left( \sum_{s=1}^t x_s x_s^\top \right)^{\dagger},
\end{equation}
which we can re-write as (using the notation introduced in the introduction)
\begin{equation}\label{eq:estimation error matrix form}
    A_{t} - A = E^\top X (X^\top X)^\dagger.
\end{equation}

Before stating our main result, we define a quantity that will arise naturally in our analysis. We define the truncated block Toeplitz matrix $\Gamma$ as
\begin{equation}
  \Gamma = \begin{bmatrix}
  I_d\phantom{^{t-1}} &        &        &     &    \\
  A\phantom{^{t-1}} & I_d      &        &  O  &    \\
                    &        & \ddots &     &    \\
  A^{t-2}           & \hdots &   A    &  I_d  &    \\
  A^{t-1}           & \hdots & \hdots &  A  & I_d  \\
  \end{bmatrix}.
\end{equation}
When $\rho(A)<1$, we can upper bound $\Vert \Gamma \Vert$ by a quantity ${\cal J}(A)$, independent of $t$, as shown in Lemma \ref{lem4} presented in the appendix:
\begin{equation}\label{eq:toeplitz}
    \Vert \Gamma \Vert \le {\cal J}(A)=\sum_{s\ge 0} \|A^s\|.
\end{equation}

\medskip
\noindent
Our main result is the following Theorem.
\medskip

\begin{theorem}[OLS performance]
Let $A_t$ denote the OLS estimator. For any $0<\delta<1$, and any $\varepsilon >0 $, we have:
$
\mathbb{P}\left( \Vert A_{t} - A \Vert > \varepsilon \right) < \delta,
$
as long as the following condition holds
\begin{equation}\label{eq:sample complexity upper bound}
  \lambda_{\min}\left(\sum_{s=0}^{t-1} \Gamma_s (A)\right) \ge  c \max\left\{{1\over \varepsilon^2},\|\Gamma \|^2\right\} \left(\log ( \frac{1}{\delta}) + d \right),
\end{equation}
where $c=c'K^4$ and $c'>0$ is a universal constant.
\label{th1}
\end{theorem}

%% file: preliminaries.tex

\section{Spectrum of the Covariates Matrix}

In this section, we analyze the spectrum of the covariates matrix $X$. Such an analysis is at the core of proof of Theorem \ref{th1}. Indeed, relating $s_d(X)$ to the spectrum of the matrix $M=\left(\sum_{s=0}^{t-1} \Gamma_s(A)\right)^{-\frac{1}{2}}$ is close to what Theorem \ref{th1} states. We are actually able to control the entire spectrum of $X$ with high probability, as stated in Theorem \ref{thm:spectrum deviations} below. We believe that this result is of independent interest; but it is not directly used in the poof of Theorem \ref{th1}. The latter relies on Lemma \ref{lem2}, which is also the main ingredient of the proof of Theorem \ref{thm:spectrum deviations}.

\begin{theorem}\label{thm:spectrum deviations}
Let $\varepsilon > 0$. Let $M=\left(\sum_{s=0}^{t-1} \Gamma_s(A)\right)^{-\frac{1}{2}}$. Then:
$$\frac{1}{\Vert M \Vert}(1 - K^2\varepsilon) \le s_d(X) \le \dots \le s_1(X) \le (1 + K^2\varepsilon)  \frac{1}{s_d(M)}
$$
holds with probability at least
$$
1 - 2\exp \left(- c_1 \varepsilon^2 \frac{1}{\Vert M \Vert^2 \Vert \Gamma \Vert^2 } + c_2 d\right),
$$
for some universal constants $c_1, c_2>0$.
\end{theorem}

Theorem \ref{thm:spectrum deviations} is a direct consequence of the following two lemmas. Lemma \ref{lem:approximate isometry} is a corollary of \cite[Ch. 4, Lemma 4.5.6]{vershynin_2018}), and is proved in the appendix. Informally, in this lemma, we may think of $X$ as a tall random matrix and $M^{-1}$ as a deterministic normalizing matrix. If the latter is chosen appropriately, then one would hope that $XM$ is an approximate isometry\footnote{ A mapping $T: \mathcal{X} \to \mathcal{Y}$ is said to be an isometry if $d_{\mathcal{X}}(x, y) = d_{\mathcal{Y}}(Tx, Ty)$ for all $x, y \in \mathcal{X}$.} in the sense of Lemma \ref{lem:approximate isometry}, which in turn will provide a two-sided bound on the spectrum of $X$.

\begin{lemma}[Approximate isometries]\label{lem:approximate isometry}
Let $X\in \mathbb{R}^{t\times d}$, and let $M\in \mathbb{R}^{d\times d}$ be a full rank matrix. Let $\varepsilon>0$ and assume that
\begin{equation}\label{eq:approximate isometry}
    \Vert (XM)^\top XM - I_d \Vert \le \max(\varepsilon, \varepsilon^2).
\end{equation}
Then the following holds
$$
\frac{1}{s_1(M)}(1 - \varepsilon) \le s_d(X) \le \dots \le s_1(X) \le (1 + \varepsilon)  \frac{1}{s_d(M)}.
$$
\end{lemma}

\medskip
Lemma \ref{lem2} is the main ingredient of the proof of Theorem \ref{thm:spectrum deviations}, but also of that of Theorem \ref{th1}. The lemma is established by expressing $\Vert (XM)^\top XM - I_d \Vert$ as the supremum of a {\it chaos process}, which we can control combining Hanson-Wright inequality and a classical $\epsilon$-net argument. In the next two subsections, we formally introduce chaos processes and provide related concentration inequalities; we also present the $\epsilon$-net argument used to complete the proof of Lemma \ref{lem2}.

\begin{lemma}\label{lem2}
Let $M=\left(\sum_{s=0}^{t-1} \Gamma_s(A)\right)^{-\frac{1}{2}}$.
$$
\Vert (XM)^\top XM - I_d \Vert > \max(\varepsilon, \varepsilon^2) K^2
$$
holds with probability at most
$$
2 \exp \left(- c_1 \varepsilon^2 \frac{1}{\Vert M \Vert^2 \Vert \Gamma \Vert^2 } + c_2 d\right)
$$
for some positive absolute constants $c_1, c_2$.
\end{lemma}

\subsection{Chaos Processes and Hanson-Wright Inequality}

In this section, we introduce chaos processes, and provide Hanson-Wright inequality, an instrumental concentration result related to these processes.
A \emph{chaos} in probability theory refers to a quadratic form $ \xi^\top W \xi$, where $W$ is a deterministic matrix in $\mathbb{R}^{d\times d}$ and $\xi$ is a random vector with independent coordinates. If the vector $\xi$ is isotropic, we simply have $\E \xi^\top W \xi = \textrm{tr}W$. The process $( \xi^\top W \xi)_{W\in {\cal W}}$ indexed by a set ${\cal W}$ of matrices is referred to as a {\it chaos process}.

In our analysis, we use the following concentration result on chaos, due to Hanson-Wright \cite{hanson1971bound,wright1973bound}, see \cite{rudelson2013hanson} (Theorems 2.1).

\begin{theorem}[Concentration of anisotropic random  vectors]\label{thm:ARV}
Let $B \in \mathbb{R}^{m \times d}$, and $\xi \in \mathbb{R}^d$ be a random vector with zero-mean, unit-variance, sub-gaussian independent coordinates. Then for all $\varepsilon > 0$,
\begin{align*}
  \mathbb{P}&\left[\left \vert \Vert B\xi \Vert_2^2 - \Vert B \Vert_F^2  \right\vert \right. \left.  > \varepsilon \Vert B  \Vert_F^2   \right]   \\ &  \qquad  \qquad  \le 2 \exp \left(- c \min\left(\frac{\varepsilon^2}{K^4}, \frac{\varepsilon}{K^2}\right) \frac{\Vert B \Vert_F^2}{\Vert B \Vert^2} \right),
\end{align*}
where $c$ is an absolute positive constant and $K = \Vert \xi \Vert_{\psi_2}$.
\end{theorem}

\subsection{The $\epsilon$-net argument}\label{subsec:eps}

The epsilon net argument is a simple, yet powerful tool in the non-asymptotic theory of random matrices and high dimensional probability \cite{litvak2005smallest, vershynin_2018, tao2012topics}. Here, we provide two instances of this argument.
\begin{lemma}\label{lem:net argument 1}
  Let $W$ be an $m \times d$ random matrix, and $\epsilon \in [0, 1)$. Furthermore, let $\mathcal{N}$ be an $\epsilon$-net of $S^{d-1}$ with minimal cardinality. Then for all $\rho > 0$, we have
  \begin{align*}
      \mathbb{P}\left( \Vert W \Vert > \rho \right) & \le \left( \frac{2}{\epsilon} + 1\right)^d \max_{x \in \mathcal{N}} \mathbb{P}\left( \Vert W x \Vert_2 > (1-\epsilon)\rho \right).
  \end{align*}
\end{lemma}
\medskip 
\begin{lemma}\label{lem:net argument 2}
  Let $W$ be an $d \times d$ a symmetric random matrix, and $\epsilon \in [0, 1/2)$. Furthermore, let $\mathcal{N}$ be an $\epsilon$-net of $S^{d-1}$ with minimal cardinality. Then for all $\rho > 0$, we have
  \begin{align*}
      \mathbb{P}\left( \Vert W \Vert > \rho \right) & \le \left( \frac{2}{\epsilon} + 1\right)^d \max_{x \in \mathcal{N}} \mathbb{P}\left( \vert x^\top  W x \vert > (1-2 \epsilon)\rho \right).
  \end{align*}
\end{lemma}
\medskip
\noindent
Lemma \ref{lem:net argument 1} and Lemma \ref{lem:net argument 2} exploit variational forms of the operator norm, namely $
\Vert W \Vert = \sup_{x \in S^{d-1}} \Vert W x \Vert_2
$ and, when $W$ is symmetric, $ \Vert W \Vert = \sup_{x \in S^{d-1}} \vert x^\top W x \vert$. The proof of theses standard lemmas can be found for example in \cite[Chapter 4]{vershynin_2018}.


%

\subsection{Proof of Lemma \ref{lem2}}

\noindent
\textbf{Step 1} ($\Vert (XM)^\top XM - I_d \Vert$ as the supremum of a chaos process)\\
Note that we can write $\textrm{vec}(X^\top) = \Gamma \xi$.
It follows from the fact that $\xi$ is isotropic that
$$
\E X^\top X  = \E \sum_{s=1}^{t} x_s x_s^\top = \sum_{s=1}^{t} \sum_{k=0}^{s-1} A^k (A^k)^\top = \sum_{s=0}^{t-1} \Gamma_s(A).
$$
Noting that $\E X^\top X$ is a symmetric positive definite matrix, we may define the inverse of its positive definite symmetric square root matrix
$M= \left(\sum_{s=0}^{t-1} \Gamma_s(A)\right)^{-\frac{1}{2}}$. We think of the inverse of this matrix as a normalization of $X$. Now, we have:
\begin{align}
    \Vert  (XM&)^\top XM - I_d \Vert   \nonumber \\
    & =  \sup_{u\in S^{d-1}} \left \vert u^\top \left( (XM)^\top X M  -   I_d \right) u \right\vert    \label{eq:suprima1}  \\
    & = \sup_{u\in S^{d-1}} \left \vert u^\top (XM)^\top X M u -  \E u^\top (XM)^\top X M u \right\vert \nonumber \\
    & = \sup_{u\in S^{d-1}} \left \vert \Vert X M u \Vert_2^2 -   \E \Vert X M u \Vert_2^2 \right\vert. \nonumber \\
     & = \sup_{u\in S^{d-1}} \left \vert \Vert  {\sigma_{Mu}}^\top \Gamma \xi \Vert_2^2 - \E \Vert  {\sigma_{Mu}}^\top \Gamma \xi \Vert_2^2  \right\vert \nonumber \\
   & = \sup_{u\in S^{d-1}} \left \vert \Vert  {\sigma_{Mu}}^\top \Gamma \xi \Vert_2^2 - \Vert  {\sigma_{Mu}}^\top \Gamma \Vert_F^2  \right\vert. \label{eq:suprima2}
\end{align}
where in the first equality, we used the variational form of the operator norm for symmetric matrices,  where in fourth equality, the $td\times t$ matrix $\sigma_{Mu}$ is defined as
$$
\sigma_{Mu} =
\begin{bmatrix}
Mu &   &        & O \\
  & Mu &        &  \\
  &    & \ddots &  \\
 O &   &        & Mu
\end{bmatrix},
$$
and the last equality follows from the fact that $\xi$ is isotropic. In fact, by definition of $M$, we have $\Vert  {\sigma_{Mu}}^\top \Gamma \Vert_F^2 = 1$ for all $u \in S^{d-1}$. We have proved that $\Vert (XM)^\top XM - I_d \Vert$ is the supremum of a \emph{chaos} process $(\xi^\top  W \xi)_W$, where the parametrizing matrix $W$ is
$$
W =  \Gamma^\top \sigma_{Mu}{\sigma_{Mu}}^\top \Gamma.
$$

\medskip
\noindent
\textbf{Step 2} (Uniform bound on the chaos) Let $\rho > 0$, and $u \in \mathcal{S}^{d-1}$. Again, recalling that ${\xi}$ is a zero-mean, subgaussian random vector with independent coordinates, from Hanson-Wright inequality (see Theorem \ref{thm:ARV}), we deduce that:
$$
\left \vert \Vert \sigma_{Mu}^\top \Gamma \xi \Vert_2^2 - \Vert \sigma_{Mu}^\top \Gamma \Vert_F^2  \right\vert > \rho \Vert  {\sigma_{Mu}}^\top \Gamma \Vert_F^2
$$
holds with probability at most
$$
2 \exp \left(- c \min\left(\frac{\rho^2}{K^4}, \frac{\rho}{K^2}\right) \frac{\Vert  {\sigma_{Mu}}^\top \Gamma \Vert_F^2}{ \Vert  {\sigma_{Mu}}^\top \Gamma \Vert^2 } \right)
$$
for some positive universal constant $c$. Noting that $\Vert  {\sigma_{Mu}}^\top \Gamma \Vert_F^2 = 1$, and $\Vert  {\sigma_{Mu}}^\top \Gamma \Vert \le \Vert M \Vert \Vert \Gamma \Vert$, an upper bound on the above probability is
$$
2 \exp \left(- c \min\left(\frac{\rho^2}{K^4}, \frac{\rho}{K^2}\right) \frac{1}{ \Vert M \Vert^2 \Vert \Gamma \Vert^2 } \right).
$$
\medskip
\noindent
\textbf{Step 3} ($\epsilon$-net argument) Recalling the equalities \eqref{eq:suprima1} and \eqref{eq:suprima2} in step 1, as a consequence of step 2, we have: for all $\epsilon\in [0, 1/2)$, for all $u \in \mathcal{S}^{d-1}$, the following
\begin{align*}
 \left \vert u^\top \left((XM)^\top X M - I_d\right) u \right\vert > (1-2\epsilon)\rho
\end{align*}
holds with probability at most
$$
2 \exp \left(- c \min\left(\frac{(1-2\epsilon)^2\rho^2}{K^4}, \frac{(1-2\epsilon)\rho}{K^2}\right) \frac{1}{ \Vert M \Vert^2 \Vert \Gamma \Vert^2 } \right).
$$
Now choosing $\epsilon = \frac{1}{4}$ and applying Lemma \ref{lem:net argument 2}, we obtain that
$$
\Vert (XM)^\top XM - I_d \Vert > \rho
$$
holds with probability at most
$$
2 \cdot 9^d\exp \left(- c \min\left(\frac{\rho^2}{K^4}, \frac{\rho}{K^2}\right) \frac{1}{ 4 \Vert M \Vert^2 \Vert \Gamma \Vert^2 } \right),
$$
where we used $\min\left(\frac{\rho^2}{4K^4}, \frac{\rho}{2K^2}\right) \ge \frac{1}{4}\min\left(\frac{\rho^2}{K^4}, \frac{\rho}{K^2}\right)$.  By choosing $\rho = \max(\varepsilon, \varepsilon^2)K^2$, which is equivalent to $\varepsilon^2 = \min(\frac{\rho}{K^2}, \frac{\rho^2}{K^4})$, we obtain that
$$
\Vert (XM)^\top XM - I_d \Vert > \max(\varepsilon, \varepsilon^2)K^2
$$
holds with probability at most
$$
2 \exp \left(- c_1 \varepsilon^2 \frac{1}{ \Vert M \Vert^2 \Vert \Gamma \Vert^2 } + c_2 d\right),
$$
for some positive absolute constants $c_1, c_2$. This completes the proof of Lemma \ref{lem2}.

%% file: proof_new.tex
\section{Proof of Theorem \ref{th1}}

We first decompose the estimation error as:
$$
\Vert A_{t} - A \Vert \le \Vert E^\top X ((X^\top X)^\dagger)^\frac{1}{2} \Vert \Vert ((X^\top X)^\dagger)^\frac{1}{2} \Vert.
$$
We introduce the matrix $M =  \left(  \sum_{s=1}^{t-1} \Gamma_s(A) \right)^{-\frac{1}{2}}$, and the events ${\cal E}_1$ and $\mathcal{E}_2$ defined as:
\begin{align*}
  \mathcal{E}_1 &= \left\lbrace \Vert E^\top X ((X^\top X)^\dagger)^\frac{1}{2} \Vert \Vert ((X^\top X)^\dagger)^\frac{1}{2} \Vert > \varepsilon \right \rbrace \\
  \mathcal{E}_2 &= \left\lbrace  \Vert (XM)^\top XM - I_d \Vert \le {1\over 2} \right\rbrace.
\end{align*}
Observe that:
\begin{align*}
    \mathbb{P} & \left( \Vert A_{t} - A \Vert > \varepsilon \right) \le \mathbb{P}\left( \mathcal{E}_1 \cap \mathcal{E}_2 \right) + \mathbb{P}\left( \overline{\mathcal{E}_2}\right).
\end{align*}

\subsection{Upper bound of $\mathbb{P}\left( \overline{\mathcal{E}_2}\right)$}
We use Lemma \ref{lem2}. Let $\rho = \max(\varepsilon^2, \varepsilon)K^2$, which is equivalent to writing $\varepsilon^2 = \min\left(\frac{\rho}{K^2}, \frac{\rho^2}{K^4}\right)$. Then choose $\rho = {1 \over 2}$. With this choice, Lemma \ref{lem2} implies that
\begin{equation*}
  \Vert (XM)^\top XM - I_d \Vert \le \frac{1}{2}
\end{equation*}
holds with probability at most
$$
2 \exp \left(- c_1 \min\left(\frac{1}{2K^2}, \frac{1}{4K^4}\right) \frac{1}{\Vert M \Vert^2 \Vert \Gamma \Vert^2 } + c_2 d\right).
$$
Since $\xi$ is sub-gaussian and isotropic, we have $K \ge 1$. We conclude that $\mathbb{P}\left( \overline{\mathcal{E}} \right) \le \frac{\delta}{2}$ when
\begin{equation}\label{condition1}
    \frac{1}{\Vert M \Vert^2} \ge \frac{16K^4 \Vert \Gamma \Vert^2}{c} \left(\log\left(\frac{4}{\delta}\right) + d \log(9) \right).
\end{equation}

\subsection{Upper bound of $\mathbb{P}\left( \mathcal{E}_1 \cap \mathcal{E}_2 \right)$}

We derive an upper bound on the above probability using a similar technique as in \cite{sarkar2019near} (see also \cite{gonzlez2019finitesample} where similar decompositions are used for the analysis of autoregressive processes)�. We first use the event ${\cal E}_2$ to simplify the condition $\left\lbrace \Vert E^\top X ((X^\top X)^\dagger)^\frac{1}{2} \Vert \Vert ((X^\top X)^\dagger)^\frac{1}{2} \Vert > \varepsilon \right \rbrace$ until we get a quantity that can be analyzed using concentration results on self-normalized processes.

When the event $\mathcal{E}_2$ occurs, we have:
$$
 \frac{3}{2} I_d \succeq (XM)^\top XM  \succeq \frac{1}{2} I_d
$$
or equivalently
$$
 \frac{3}{2} \sum_{s=0}^{t-1} \Gamma_s(A) = \frac{3}{2} M^{-2} \succeq X^\top X \succeq \frac{1}{2} M^{-2} = \frac{1}{2} \sum_{s=0}^{t-1} \Gamma_s(A).
$$
Define $S = \frac{1}{2}M^{-2}$ and $\beta = \sqrt{s_d(S)}$.
Note that when event $\mathcal{E}_2$ occurs,
$$
s_d(X) \ge  \beta = \sqrt{\frac{1}{2}\lambda_{\min}\left( \sum_{s=0}^{t-1} \Gamma_s(A) \right)}.
$$
Thus $\Vert X^\dagger \Vert \le \frac{1}{\beta}$ and we obtain:
$$
\mathcal{E}_1 \cap \mathcal{E}_2 \subseteq \lbrace  \Vert E^\top X (X^\top X)^{-\frac{1}{2}} \Vert > \varepsilon \beta \rbrace \cap\mathcal{E}_2.
$$
Next , again when $\mathcal{E}_2$ occurs, we have $2 X^\top X \succeq X^\top X + S$, and thus
$2 (X^\top X + S)^{-1} \succeq (X^\top X)^{-1}$. We deduce that:
\begin{align*}
\Vert E^\top X (X^\top X)^{-\frac{1}{2}} \Vert \le \sqrt{2} \Vert E^\top X (X^\top X + S )^{-\frac{1}{2}} \Vert.
\end{align*}
Hence,
\begin{align*}
& \lbrace \Vert E^\top X (X^\top X)^{-\frac{1}{2}} \Vert > \varepsilon \rbrace \cap\mathcal{E}_2  \\ & \qquad \qquad \subseteq \left \lbrace \sqrt{2} \Vert E^\top X (X^\top X + S )^{-\frac{1}{2}} \Vert > \varepsilon \beta \right \rbrace \cap\mathcal{E}_2.
\end{align*}
Furthermore, note that under $\mathcal{E}_2$, we have  $3 S \succeq X^\top X \succeq S$, which implies that for all $\delta \in (0, 1)$,
$$
\varepsilon \beta \ge \varepsilon \beta \sqrt{ \frac{\log \left( \frac{2\cdot 5^d \left(\det(X^\top X + S) S^{-1}\right)^{\frac{1}{2}}}{ \delta}\right)}{\log \left( \frac{2\cdot5^d \left(\det(3S + S) S^{-1}\right)^{\frac{1}{2}}}{ \delta}\right)}}.
$$
Now consider the condition
\begin{align*}
  \varepsilon & \ge  \frac{4\sqrt{2}\sqrt{c}K}{\beta} \sqrt{\log \left( \frac{2 \cdot 5^d \left(\det(3S + S) S^{-1}\right)^{\frac{1}{2}}}{ \delta}\right)^{-1}} \\
  & = \frac{ 4\sqrt{2}\sqrt{c}K}{\beta} \sqrt{\left( \log \left( \frac{2}{ \delta}\right) + d \log(10)\right)^{-1}}.
\end{align*}
Since $\beta^2 = \frac{1}{2\Vert M \Vert^2}$, the above condition is equivalent to
\begin{equation}\label{condition2}
    \frac{1}{\Vert M \Vert^2} \ge \frac{16 c K^2}{\varepsilon^2}\left( \log\left(\frac{2}{\delta} \right) + d \log(10)\right)
\end{equation}
We deduce that, under the above condition,
\begin{align*}
  & \Big\lbrace \sqrt{2} \Vert  E^\top X  (X^\top X + S )^{-\frac{1}{2}}  \Vert > \varepsilon \beta  \Big \rbrace \cap\mathcal{E}_2 \\
  & \subseteq \Bigg \lbrace \Vert E^\top X (X^\top X + S )^{-\frac{1}{2}} \Vert \\
  & > 4\sqrt{c}K   \sqrt{  \log \left( \frac{2\cdot 5^d \left(\det(X^\top X + S) S^{-1}\right)^{\frac{1}{2}}}{ \delta}\right)} \Bigg \rbrace.
\end{align*}
We are now ready to apply the results presented in the appendix in Corollary \ref{cor1} for self-normalized processes, and conclude that the event
\begin{align*}
& \Bigg \lbrace \Vert E^\top X (X^\top X + S )^{-\frac{1}{2}} \Vert  \\
& > 4\sqrt{c}K   \sqrt{  \log \left( \frac{2\cdot 5^d \left(\det(X^\top X + S) S^{-1}\right)^{\frac{1}{2}}}{ \delta}\right)} \Bigg \rbrace
\end{align*}
occurs with probability at most $\delta/2$. This implies that:
$$
\mathbb{P}\left( \Big\lbrace \sqrt{2} \Vert E^\top X  (X^\top X + S )^{-\frac{1}{2}}  \Vert > \varepsilon \beta  \Big \rbrace \cap\mathcal{E}_2 \right) \le \frac{\delta}{2}
$$
as long as condition $\eqref{condition2}$ holds. Hence
$
\mathbb{P}\left( \mathcal{E}_1  \cap \mathcal{E}_2 \right) \le \frac{\delta}{2}
$
when (\ref{condition2}) holds.

\subsection{Concluding steps}
Introducing
\begin{align*}
  \tau_1 & = \frac{16K^4\Vert \Gamma \Vert^2}{c_2} \left(\log\left(\frac{4}{\delta}\right) + d \log(9) \right), \\
  \tau_2 & = \frac{16 c_1 K^2}{\varepsilon^2}\left( \log\left(\frac{2}{\delta} \right) + d \log(10)\right), \\
\end{align*}
we may re-write Condition \eqref{condition1} as $\frac{1}{\Vert M \Vert^2} \ge \tau_1$ and Condition \eqref{condition2} as $\frac{1}{\Vert M \Vert^2} \ge \tau_2$. Thus, we have
\begin{align*}
\mathbb{P}\left(  \Vert A_{t+1} - A \Vert > \varepsilon \right) & \le \mathbb{P}\left( \mathcal{E}_1 \cap \mathcal{E}_2\right) + \mathbb{P}\left( \overline{\mathcal{E}_2}\right) &  \le \frac{\delta}{2} + \frac{\delta}{2}
\end{align*}
provided that
$
\frac{1}{\Vert M\Vert^2} \ge \max(\tau_1, \tau_2).
$
Finally observe that:
\begin{align*}
  & \max(\tau_1, \tau_2)  \\ & \quad \le   16 K^2 \max  \left( \frac{c_1}{\varepsilon^2},\frac{K^2 }{c_2} \Vert \Gamma \Vert^2  \right) \left( \log\left(\frac{4}{\delta} \right)  + d \log(10)\right) \\
  & \quad  \le C \max\left(\frac{1}{\varepsilon^2}, \Vert \Gamma \Vert^2 \right)\left( \log\left(\frac{4}{\delta}  \right) + d \log(10) \right)
\end{align*}
where $C = 16 K^2 \max\left(c_1, \frac{K^2}{c_2}\right)$. We conclude that
$$
\mathbb{P}\left(  \Vert A_{t+1} - A \Vert > \varepsilon \right) \le \delta
$$
holds as long as
$$
 \frac{1}{\Vert M \Vert^2} \ge C  \max\left(\frac{1}{\varepsilon^2}, \Vert \Gamma \Vert^2 \right)\left( \log\left(\frac{4}{\delta} \right) + d \log(10) \right).
$$
This completes the proof of Theorem \ref{th1}.

%% file: appA.tex

\section{Proof of Lemma \ref{lem:approximate isometry}}

\begin{proof}
Observe that
  $$
  \Vert (XM)^\top XM - I_d \Vert = \sup_{u \in {S^{d-1}}} \left\vert \Vert XMu \Vert_2^2 - 1  \right\vert.
  $$
  Using the fact that $\vert z^2 - 1 \vert \ge \max(\vert z-1 \vert, \vert z-1 \vert^2)$ for all $z > 0$, we obtain for all $u \in S^{d-1}$
  \begin{align*}
    \max(\left\vert \Vert XMu \Vert_2 - 1  \right\vert, \left\vert \Vert XMu \Vert_2 - 1  \right\vert^2) \le \max(\varepsilon, \varepsilon^2).
  \end{align*}
By the monotonicity of $z \mapsto \max(z, z^2)$ for $z \ge  0$, we deduce that:
\begin{align*}
    \left \vert \Vert X M u \Vert_2 - 1 \right\vert \le \varepsilon,
  \end{align*}
  from which we conclude
  \begin{align*}
    1 - \varepsilon \le \Vert XM u \Vert_2 \le 1 + \varepsilon.
  \end{align*}
  It follows immediately that:
  \begin{align*}
    1-\varepsilon  \le s_d(XM) \le \dots \le s_1(XM)\le 1+\varepsilon.
  \end{align*}
  The proof is completed observing that:
  $$
  s_d(X M) \le s_d(X) s_1(M) \quad \textrm{and} \quad   s_1(X) s_d(M)\le s_1(XM).
  $$
\end{proof}

\section{Properties of $\Gamma$}

\begin{lemma}\label{lem4}
  Consider the block Toeplitz matrix $\Gamma$ as defined earlier, and assume that $\rho(A) < 1$. Then, there exists a positive constant $\mathcal{J}(A)$ that only depends on $A$, such that for all $t \ge 1$,
  $
   \Vert \Gamma \Vert \le \mathcal{J}(A).
  $
\end{lemma}

\begin{proof}
  From the theory of Toeplitz matrices (see \cite{bottcher2012introduction}), we have
  $$
  \Vert \Gamma \Vert  \le \sup_{x \in [ 0, 1]} \left \Vert \sum_{s=0}^{t-1} A^s e^{2\pi i s x} \right \Vert.
  $$
  The argument behind the above inequality consists of noting that $\Gamma$ is a submatrix of the infinite banded block Toeplitz matrix $\mathcal{T}(A)$ defined as
  \begin{equation*}
    \mathcal{T}(A) = \begin{bmatrix}
      I_d     &          &         &         \\
      A       &  \ddots  &         &         \\
      \vdots  &  \ddots  & I_d     &         \\
      A_{t-1} &          & A       &  \ddots \\
              &  \ddots  & \vdots  &  \ddots \\
              &          & A_{t-1} &         \\
              &          &         &  \ddots
  \end{bmatrix},
  \end{equation*}
  thus $\Vert \Gamma \Vert \le \Vert \mathcal{T}(A)\Vert$. It follows immediately from the results in \cite[chapter 6]{bottcher2012introduction} that
  $$
  \Vert \mathcal{T}(A) \Vert = \sup_{x \in [ 0, 1]} \left \Vert \sum_{s=0}^{t-1} A^s e^{2\pi i s x} \right \Vert.
  $$
Next, note that:
  \begin{align*}
          \sup_{x \in [ 0, 1]}  \left \Vert \sum_{s=0}^{t-1} A^s e^{2\pi i s x} \right \Vert &  \le  \sup_{x \in [ 0, 1]}   \sum_{s=0}^{t-1}  \left \Vert  A^s e^{2\pi i s x}  \right \Vert \\
          & \le  \sum_{s=0}^{t-1} \Vert  A^s \Vert
  \end{align*}
  Classic results on convergence of Matrix power series ensure convergence of $\sum_{s=0}^{t-1} \Vert  A^s \Vert$, under the condition that $\rho(A) < 1$ (in particular, $\Vert A^s \Vert$ is asymptotically equivalent to $\rho(A)^s$, see for instance \cite{young1981therate}). Hence if $\mathcal{J}(A)=\sum_{s=0}^{\infty} \Vert  A^s \Vert$, we have
  $
    \Vert \Gamma \Vert \le \mathcal{J}(A).
  $
\end{proof}

\section{Concentration results for self-normalized processes}

Proposition \ref{prop:snp} can be found in   \cite{abbasi2011improved}. It is an application of the theory of self-normalized processes \cite{pena2008book}.
\begin{proposition}[Self-normalized concentration]\label{prop:snp}
Let $\lbrace \mathcal{F}_{t} \rbrace_{t\ge 1}$ be a filtration. Let $\lbrace \eta_t \rbrace_{t\ge 1}$ be stochastic process adapted to $\lbrace \mathcal{F}_{t} \rbrace_{t\ge 1}$ and taking values in $\mathbb{R}$. Let $\lbrace x_t \rbrace_{t\ge 1}$ be a predictable stochastic process with respect to $\lbrace \mathcal{F}_{t} \rbrace_{t\ge 1}$, taking values in $\mathbb{R}^d$. Furthermore, assume that $\eta_{t+1}$, conditionally on ${\cal F}_t$, is a sub-gaussian r.v. with $\psi_2$-norm equal to K, and let $c>0$ denote a constant such that for all $\lambda>0$
$$
\E [ \exp(\lambda \eta_{t+1}) \vert \mathcal{F}_{t} ]\le \exp (c\lambda^2K^2).
$$
Let $S$ be an $d\times d$ positive definite matrix. Using the notation $\eta^\top = \begin{bmatrix} \eta_2 & \dots & \eta_{t+1} \end{bmatrix}$ and $X^\top = \begin{bmatrix} x_1 & \dots & x_{t} \end{bmatrix}$, the following
\begin{align*}
& \left \Vert \left( X^\top X + S \right)^{-\frac{1}{2}} \left(X^\top \eta \right) \right\Vert_2^2 \\
& \qquad \qquad \le 4 c K^2 \log\left(  \frac{(\det \left( (X^\top X + S)S^{-1}\right))^{\frac{1}{2}}}{\delta} \right) .
\end{align*}
holds with probability at least $1 - \delta$.
\label{prop1}
\end{proposition}
\noindent
The following result follows from a $\varepsilon$-net argument and can be found in \cite{sarkar2019near}.
\begin{corollary}\label{cor1}
Under the same assumptions as in Proposition \ref{prop1} with the exception that $\lbrace \eta_t \rbrace_{t\ge1}$ is taking values in $\mathbb{R}^d$ and that for all $\lambda > 0$, for all $u\in S^{d-1}$
$$
\E [ \exp(\lambda \eta_{t+1}^\top u ) \vert \mathcal{F}_{t} ]\le \exp (c\lambda^2K^2)
$$
for some absolute constant $c > 0$. That is to say that $\eta_{t+1}$ is, conditionally on $\mathcal{F}_t$, a subgaussian random vector with $\psi_2$-norm equal to $K$. Then for all $\delta \in (0, 1)$, the following
\begin{align*}
& \left \Vert  \left( X^\top X + S \right)^{-\frac{1}{2}}  \left(X^\top E \right) \right\Vert^2 \\
&\ \ \ \ \ \ \le 16cK^2 \left(\log\left(  \frac{5^d \left( \det \left( (X^\top X + S)S^{-1}\right)\right)^{\frac{1}{2}} }{\delta} \right)\right).
\end{align*}
holds with probability at least $1 - \delta$.
\end{corollary}
\begin{proof}
Let $\delta_0 \in (0, 1)$ to be chosen later. Let $\mathcal{N}$ be a $1/2$-net of the unit sphere $S^{d-1}$ of minimal cardinality. Consider the events
\begin{align*}
  \mathcal{E}_{\delta_0} &= \Bigg \lbrace  \left\Vert  \left( X^\top X + S \right)^{-\frac{1}{2}} \left(X^\top E \right) \right\Vert^2  \\
  & \qquad  > 16c K^2  \log\left(  \frac{\left( \det \left( (X^\top X + S)S^{-1}\right) \right)^{-\frac{1}{2}} }{\delta_0} \right) \Bigg\rbrace, \\
  \mathcal{E}_{\delta_0, u} &= \Bigg \lbrace  \left\Vert  \left( X^\top X + S \right)^{-\frac{1}{2}} \left(X^\top E \right)u \right\Vert^2  \\
  & \qquad  > 4c K^2  \log\left(  \frac{\left( \det \left( (X^\top X + S)S^{-1}\right) \right)^{-\frac{1}{2}} }{\delta_0} \right) \Bigg\rbrace,
\end{align*}
for $u \in \mathcal{N}$. The net argument (see Lemma \ref{lem:net argument 1}) implies that
$$
\mathbb{P}(\mathcal{E}_{\delta_0}) \le   5^d  \max_{u \in \mathcal{N}} \mathbb{P}(\mathcal{E}_{\delta_0, u}).
$$
By Proposition \ref{prop1}, we obtain for all $u \in \mathcal{N}$,
$$
\mathbb{P}(\mathcal{E}_{\delta_0, u}) \le \delta_0.
$$
Hence,
$$
\mathbb{P}(\mathcal{E}_{\delta_0}) \le   5^d \delta_0.
$$
Finally, substituting with $\delta = 5^{d} \delta_0$, for all $\delta_0 \in (0, 5^{-d})$, yields the desired result.
\end{proof}

%% file: main.bbl
\begin{thebibliography}{10}
\providecommand{\url}[1]{#1}
\csname url@samestyle\endcsname
\providecommand{\newblock}{\relax}
\providecommand{\bibinfo}[2]{#2}
\providecommand{\BIBentrySTDinterwordspacing}{\spaceskip=0pt\relax}
\providecommand{\BIBentryALTinterwordstretchfactor}{4}
\providecommand{\BIBentryALTinterwordspacing}{\spaceskip=\fontdimen2\font plus
\BIBentryALTinterwordstretchfactor\fontdimen3\font minus
  \fontdimen4\font\relax}
\providecommand{\BIBforeignlanguage}[2]{{%
\expandafter\ifx\csname l@#1\endcsname\relax
\typeout{** WARNING: IEEEtran.bst: No hyphenation pattern has been}%
\typeout{** loaded for the language `#1'. Using the pattern for}%
\typeout{** the default language instead.}%
\else
\language=\csname l@#1\endcsname
\fi
#2}}
\providecommand{\BIBdecl}{\relax}
\BIBdecl

\bibitem{simchowitz2018learning}
M.~Simchowitz, H.~Mania, S.~Tu, M.~I. Jordan, and B.~Recht, ``Learning without
  mixing: Towards a sharp analysis of linear system identification,'' in
  \emph{Conference On Learning Theory}, 2018, pp. 439--473.

\bibitem{jedra2019sample}
Y.~Jedra and A.~Proutiere, ``Sample complexity lower bounds for linear system
  identification,'' in \emph{IEEE Conference on Decision and Control}, 2019.

\bibitem{Goodwin:1977:DSI}
G.~C. Goodwin and R.~L. Payne, \emph{Dynamic system identification: experiment
  design and data analysis}.\hskip 1em plus 0.5em minus 0.4em\relax Academic
  press New York, 1977, vol. 136.

\bibitem{Ljung:c1}
L.~{Ljung}, ``Consistency of the least-squares identification method,''
  \emph{IEEE Transactions on Automatic Control}, vol.~21, no.~5, pp. 779--781,
  October 1976.

\bibitem{Ljung:c2}
\BIBentryALTinterwordspacing
L.~Ljung, ``On the consistency of prediction error identification methods,'' in
  \emph{System Identification Advances and Case Studies}, ser. Mathematics in
  Science and Engineering, R.~K. Mehra and D.~G. Lainiotis, Eds.\hskip 1em plus
  0.5em minus 0.4em\relax Elsevier, 1976, vol. 126, pp. 121 -- 164. [Online].
  Available:
  \url{http://www.sciencedirect.com/science/article/pii/S0076539208608711}
\BIBentrySTDinterwordspacing

\bibitem{rantzer2018}
A.~Rantzer, ``Concentration bounds for single parameter adaptive control,'' in
  \emph{2018 Annual American Control Conference (ACC)}.\hskip 1em plus 0.5em
  minus 0.4em\relax IEEE, 2018, pp. 1862--1866.

\bibitem{faradonbeh:2018:c2}
M.~K.~S. Faradonbeh, A.~Tewari, and G.~Michailidis, ``Finite time
  identification in unstable linear systems,'' \emph{Automatica}, vol.~96, pp.
  342--353, 2018.

\bibitem{sarkar:2018}
\BIBentryALTinterwordspacing
T.~Sarkar and A.~Rakhlin, ``How fast can linear dynamical systems be learned?''
  \emph{CoRR}, vol. abs/1812.01251, 2018. [Online]. Available:
  \url{http://arxiv.org/abs/1812.01251}
\BIBentrySTDinterwordspacing

\bibitem{oymak:2018}
\BIBentryALTinterwordspacing
S.~Oymak and N.~Ozay, ``Non-asymptotic identification of {LTI} systems from a
  single trajectory,'' \emph{CoRR}, vol. abs/1806.05722, 2018. [Online].
  Available: \url{http://arxiv.org/abs/1806.05722}
\BIBentrySTDinterwordspacing

\bibitem{sarkar:2019}
\BIBentryALTinterwordspacing
T.~Sarkar, A.~Rakhlin, and M.~A. Dahleh, ``Finite-time system identification
  for partially observed {LTI} systems of unknown order,'' \emph{CoRR}, vol.
  abs/1902.01848, 2019. [Online]. Available:
  \url{http://arxiv.org/abs/1902.01848}
\BIBentrySTDinterwordspacing

\bibitem{mendelson:2014}
\BIBentryALTinterwordspacing
S.~Mendelson, ``Learning without concentration,'' in \emph{Proceedings of The
  27th Conference on Learning Theory}, ser. Proceedings of Machine Learning
  Research, M.~F. Balcan, V.~Feldman, and C.~Szepesvári, Eds., vol.~35.\hskip
  1em plus 0.5em minus 0.4em\relax Barcelona, Spain: PMLR, 13--15 Jun 2014, pp.
  25--39. [Online]. Available:
  \url{http://proceedings.mlr.press/v35/mendelson14.html}
\BIBentrySTDinterwordspacing

\bibitem{vershynin:2012}
R.~Vershynin, \emph{Introduction to the non-asymptotic analysis of random
  matrices}.\hskip 1em plus 0.5em minus 0.4em\relax Cambridge University Press,
  2012, p. 210–268.

\bibitem{pena2008self}
V.~H. Pe{\~n}a, T.~L. Lai, and Q.-M. Shao, \emph{Self-normalized processes:
  Limit theory and Statistical Applications}.\hskip 1em plus 0.5em minus
  0.4em\relax Springer Science \& Business Media, 2008.

\bibitem{krahmer2014suprema}
F.~Krahmer, S.~Mendelson, and H.~Rauhut, ``Suprema of chaos processes and the
  restricted isometry property,'' \emph{Communications on Pure and Applied
  Mathematics}, vol.~67, no.~11, pp. 1877--1904, 2014.

\bibitem{rudelson2013hanson}
M.~Rudelson, R.~Vershynin \emph{et~al.}, ``Hanson-wright inequality and
  sub-gaussian concentration,'' \emph{Electronic Communications in
  Probability}, vol.~18, 2013.

\bibitem{faradonbeh2018finite}
\BIBentryALTinterwordspacing
M.~K.~S. Faradonbeh, A.~Tewari, and G.~Michailidis, ``Finite time
  identification in unstable linear systems,'' \emph{Automatica}, vol.~96, pp.
  342 -- 353, 2018. [Online]. Available:
  \url{http://www.sciencedirect.com/science/article/pii/S0005109818303546}
\BIBentrySTDinterwordspacing

\bibitem{sarkar2019near}
\BIBentryALTinterwordspacing
T.~Sarkar and A.~Rakhlin, ``Near optimal finite time identification of
  arbitrary linear dynamical systems,'' in \emph{Proceedings of the 36th
  International Conference on Machine Learning}, ser. Proceedings of Machine
  Learning Research, K.~Chaudhuri and R.~Salakhutdinov, Eds., vol.~97.\hskip
  1em plus 0.5em minus 0.4em\relax Long Beach, California, USA: PMLR, 09--15
  Jun 2019, pp. 5610--5618. [Online]. Available:
  \url{http://proceedings.mlr.press/v97/sarkar19a.html}
\BIBentrySTDinterwordspacing

\bibitem{mendelson2014learning}
S.~Mendelson, ``Learning without concentration,'' in \emph{Conference on
  Learning Theory}, 2014, pp. 25--39.

\bibitem{vershynin_2018}
R.~Vershynin, \emph{High-Dimensional Probability: An Introduction with
  Applications in Data Science}, ser. Cambridge Series in Statistical and
  Probabilistic Mathematics.\hskip 1em plus 0.5em minus 0.4em\relax Cambridge
  University Press, 2018.

\bibitem{hanson1971bound}
D.~L. Hanson and F.~T. Wright, ``A bound on tail probabilities for quadratic
  forms in independent random variables,'' \emph{The Annals of Mathematical
  Statistics}, vol.~42, no.~3, pp. 1079--1083, 1971.

\bibitem{wright1973bound}
\BIBentryALTinterwordspacing
F.~T. Wright, ``A bound on tail probabilities for quadratic forms in
  independent random variables whose distributions are not necessarily
  symmetric,'' \emph{Ann. Probab.}, vol.~1, no.~6, pp. 1068--1070, 12 1973.
  [Online]. Available: \url{https://doi.org/10.1214/aop/1176996815}
\BIBentrySTDinterwordspacing

\bibitem{litvak2005smallest}
\BIBentryALTinterwordspacing
A.~Litvak, A.~Pajor, M.~Rudelson, and N.~Tomczak-Jaegermann, ``Smallest
  singular value of random matrices and geometry of random polytopes,''
  \emph{Advances in Mathematics}, vol. 195, no.~2, pp. 491 -- 523, 2005.
  [Online]. Available:
  \url{http://www.sciencedirect.com/science/article/pii/S0001870804002750}
\BIBentrySTDinterwordspacing

\bibitem{tao2012topics}
T.~Tao, \emph{Topics in random matrix theory}.\hskip 1em plus 0.5em minus
  0.4em\relax American Mathematical Soc., 2012, vol. 132.

\bibitem{gonzlez2019finitesample}
\BIBentryALTinterwordspacing
R.~A. González and C.~R. Rojas, ``A finite-sample deviation bound for stable
  autoregressive processes,'' 2019. [Online]. Available:
  \url{https://arxiv.org/abs/1912.08103}
\BIBentrySTDinterwordspacing

\bibitem{bottcher2012introduction}
A.~B{\"o}ttcher and B.~Silbermann, \emph{Introduction to large truncated
  Toeplitz matrices}.\hskip 1em plus 0.5em minus 0.4em\relax Springer Science
  \& Business Media, 2012.

\bibitem{young1981therate}
\BIBentryALTinterwordspacing
N.~Young, ``The rate of convergence of a matrix power series,'' \emph{Linear
  Algebra and its Applications}, vol.~35, pp. 261 -- 278, 1981. [Online].
  Available:
  \url{http://www.sciencedirect.com/science/article/pii/0024379581902780}
\BIBentrySTDinterwordspacing

\bibitem{abbasi2011improved}
Y.~Abbasi-Yadkori, D.~P{\'a}l, and C.~Szepesv{\'a}ri, ``Improved algorithms for
  linear stochastic bandits,'' in \emph{Advances in Neural Information
  Processing Systems}, 2011, pp. 2312--2320.

\bibitem{pena2008book}
V.~H. Pe{\~n}a, T.~L. Lai, and Q.-M. Shao, \emph{Self-normalized processes:
  Limit theory and Statistical Applications}.\hskip 1em plus 0.5em minus
  0.4em\relax Springer Science \& Business Media, 2008.

\end{thebibliography}
